\documentclass[12pt,dvips]{article}

\usepackage{amsmath,amsfonts,amssymb,amsthm,graphicx,verbatim,subfig}
\usepackage[all]{xy}

\ifx\pdftexversion\undefined

\usepackage[a4paper,colorlinks,link
=black,filecolor=black,citecolor=black,urlcolor=black,pdfstartview=FitH]{hyperref}
\else

\usepackage[a4paper,colorlinks,linkcolor=black,filecolor=black,citecolor=black,urlcolor=black,pdfstartview=FitH]{hyperref}
\fi

\font\sixbb=msbm6
\font\eightbb=msbm8
\font\twelvebb=msbm10 scaled 1095
\newfam\bbfam
\textfont\bbfam=\twelvebb \scriptfont\bbfam=\eightbb
                           \scriptscriptfont\bbfam=\sixbb

\newtheorem{theorem}{\bf Theorem}[section]

\newtheorem{proposition}[theorem]{\bf Proposition}

\newtheorem{definition}[theorem]{\bf Definition}

\title{On the vertices of the $d$-dimensional Birkhoff polytope}
\begin{document}
\author{Nathan Linial\thanks{Department of Computer Science, Hebrew University, Jerusalem 91904,
    Israel. e-mail: nati@cs.huji.ac.il~. Supported by ISF and BSF grants.}
  \and {Zur Luria\thanks{Department of Computer Science, Hebrew University, Jerusalem 91904,
    Israel. e-mail: zluria@cs.huji.ac.il~.}}
}

\date{}

\maketitle
\pagestyle{plain}

\begin{abstract}

Let us denote by $\Omega_n$ the {\em Birkhoff polytope} of $n \times n$ doubly-stochastic matrices.
As the Birkhoff-von Neumann theorem famously states, the vertex set of $\Omega_n$ coincides with the set of all $n \times n$ permutation matrices. 
Here we seek a higher-dimensional analog of this basic fact. Let $\Omega^{(2)}_n$ be the polytope which consists of all {\em tristochastic} arrays of order $n$. 
These are $n \times n \times n$ arrays with nonnegative entries in which every \textit{line} sums to $1$. What can be said about $\Omega^{(2)}_n$'s vertex set?
It is well-known that an order-$n$
Latin square may be viewed as a tristochastic array where every line contains
$n-1$ zeros and a single $1$ entry. Indeed, every Latin square of order $n$
is a vertex of $\Omega^{(2)}_n$, but as we show, such vertices constitute only
a vanishingly small subset of $\Omega^{(2)}_n$'s vertex set. 
More concretely, we show that the number of vertices of $\Omega^{(2)}_n$ is at least $(L_n)^{\frac{3}{2}-o(1)}$, where $L_n$ is the number of order $n$ Latin squares.

We also briefly consider similar problems concerning the polytope of $n \times n \times n$ arrays where the entries in every {\em coordinate hyperplane} sum to $1$. Several open questions are presented as well.

\end{abstract}

\section{Introduction}

Let $\Omega_n \subset \mathbb R^{n^2}$ be the Birkhoff polytope,
namely the set of order-$n$ doubly stochastic matrices.
The defining equations and inequalities of $\Omega_n$ are 
$$ \sum_{i=1}^n{x_{i,j}} = 1 \text{ for all } 1 \leq j \leq n  $$
$$ \sum_{j=1}^n{x_{i,j}} = 1 \text{ for all } 1 \leq i \leq n  $$
and 
$$ x_{i,j} \geq 0 \text{ for all } 1 \leq i,j \leq n .$$
 
The vertex set of $\Omega_n$ is determined by the Birkhoff-von Neumann theorem \cite{Birk,Neum}.

\begin{theorem}
The vertex set of $\Omega_n$ coincides with the set of permutation matrices of order $n$.
\label{BirkhoffThm}
\end{theorem}

We consider here some higher-dimensional analogs of the polytope $\Omega_n$ and ask about their vertex sets in light of Theorem~\ref{BirkhoffThm}.

A \textit{line} in an $n \times n \times n$ array $A$ is the set of entries 
obtained by fixing two indices and letting the third vary from $1$ to $n$.
A line of the form $A(\cdot,j,k)$ is called a column, a line of the form 
$A(i,\cdot,k)$ is a row and a line of the form $A(i,j,\cdot)$ is a shaft.
A \textit{coordinate hyperplane} in $A$ is the $n \times n$ matrix obtained by fixing one index and 
letting the other two vary. Such a hyperplane of the form $A(\cdot,\cdot,k)$
is called a \textit{layer} of $A$. We denote the $k$-th layer of $A$ by $A_k$.\
We denote the support of an array $A$ by $\mbox{supp}(A)$.

 Let $\Omega^{(2)}_n$ be the polytope
of all {\em tristochastic} arrays of order $n$. 
Namely, $n \times n \times n$ arrays with nonnegative entries in which every \textit{line} sums to $1$. Latin squares of order $n$ can be viewed as two-dimensional permutations and it is easily verified that every Latin square of order $n$ is a vertex of $\Omega^{(2)}_n$. Does the natural analog of Theorem~\ref{BirkhoffThm} hold true? As we show (Theorem~\ref{lower_bound}), this is far from the truth. Of the $v=v_n$ vertices of $\Omega^{(2)}_n$ only fewer than $v^{2/3+o(1)}$ correspond to Latin squares.

In section~\ref{variation} we establish a similar phenomenon for a related polytope. Namely, now we consider $n \times n \times n$ arrays of nonnegative reals in which every coordinate hyperplane sums to $1$. Again it is shown that a natural, combinatorially defined set of vertices, are a vanishingly small subset of all vertices of this polytope.

\subsection{Background material}

A Latin square $L$ of order $n$ is an $n \times n$ matrix with entries from
$[n]:=\{1, ... ,n\}$ such that each symbol appears exactly once in every row and column.
Equivalently, it is an $n \times n \times n$ array $A$ of zeros and ones
in which every \textit{line} has exactly one $1$ entry.
The correspondence between the two definitions is this:
$A(i,j,k)=1 \Leftrightarrow L(i,j)=k$. 
We denote the number of order-$n$ Latin squares by $L_n$.

The permanent of an $n \times n$ matrix $A$ is defined as
$$Per(A) = \sum_{\sigma \in \mathbb{S}_n}{\prod_{i=1}^n{a_{i,\sigma(i)}}}.$$

A lower bound on permanents of doubly stochastic matrices was conjectured by
van der Waerden and proved by Falikman and by Egorychev~\cite{Fal, Eg}.

\begin{theorem}
If $A$ is an $n\times n$ doubly stochastic matrix, then
$$ Per(A) \geq \frac{n!}{n^n} .$$
\label{VDWthm}
\end{theorem}

An upper bound on the permanent of zero/one matrices was conjectured by Minc 
and proved by Br\`{e}gman \cite{Br73}.

\begin{theorem} 
Let $A$ be an $n\times n$ matrix of zeros and ones with 
$r_i$ ones in the $i$-th row. 
Then 
$$ Per(A) \leq \prod_{i=1}^n{(r_i!)^{1/r_i}} . $$
\label{BregmanThm}
\end{theorem}

The following argument of van Lint and Wilson \cite{VL+W} utilizes
these two bounds to derive an estimate for $L_n$ by constructing a Latin square
$A$ and bounding the number of ways to do this.
Consider the $n \times n \times n$ zero-one array representation of a Latin square
layer by layer.
Each layer is a permutation matrix, so that there are $n!$ choices for the
first layer. Having already specified $k-1$ layers, the number of choices for the 
$k$-th layer can be expressed as the permanent of $B$,
a zero/one matrix where $b_{ij}=1$ iff $a_{ijt}=0$ for all $k > t$.
Using the above upper and lower bounds on $per(B)$ it follows that

\begin{theorem}
$$ L_n = \left((1+o(1))\frac{n}{e^2}\right)^{n^2} .$$
 \label{counting_Latin_squares}
\end{theorem}

\subsection{A higher dimensional Birkhoff polytope}

\subsubsection{Definitions and motivating example}\label{basic}
Let $\Omega^{(d)}_n$ be the set of $[n]^{d+1}$ nonnegative arrays such that 
the sum of each line is $1$. 
Thus, $\Omega^{(1)}_n = \Omega_n$, the set of order-$n$ doubly stochastic matrices.
Likewise, we call a member of $\Omega^{(d)}_n$
a $(d+1)$-stochastic array. Maintaining the analogy, we
let $S^{(d)}_n$ be the set of $[n]^{d+1}$ arrays of zeros and ones with a 
single one in each line. 
In other words, $S^{(d)}_n$ consists of all $(d+1)$-stochastic arrays 
all of whose entries are zero or one. 
Thus, $S^{(1)}_n$ is the set of order $n$ permutation matrices and 
$S^{(2)}_n$ coincides with the set of order-$n$ Latin squares. 
Members of $S^{(d)}_n$ are called $d$-permutations.

We turn to investigate the vertex set of $\Omega^{(d)}_n$. It is easily verified
that every member of $S^{(d)}_n$ is a vertex of $\Omega^{(d)}_n$.
However, as we show here $\Omega^{(d)}_n$ can have numerous additional vertices.
We find it instructive to present the smallest such example. 
Namely, the following array $A$ is a vertex of $\Omega^{(2)}_3$.
\[ A_1=\left[\begin{array}{ccc}
1 & 0 & 0 \\
0 & \frac{1}{2} & \frac{1}{2} \\
0 & \frac{1}{2} & \frac{1}{2} \\
\end{array}\right]
A_2=\left[\begin{array}{ccc}
0 & \frac{1}{2} & \frac{1}{2} \\
\frac{1}{2} & \frac{1}{2} &0 \\
\frac{1}{2} & 0 & \frac{1}{2} \\
\end{array}\right]
A_3=\left[\begin{array}{ccc}
0 & \frac{1}{2} & \frac{1}{2} \\
\frac{1}{2} & 0 & \frac{1}{2} \\
\frac{1}{2} & \frac{1}{2} & 0 \\
\end{array}\right]
\]

To see that $A$ is indeed a vertex, assume to the contrary that 
$A = \alpha B + (1-\alpha) C$ for some $0 < \alpha < 1$ and $B \neq C$ in $\Omega^{(2)}_3$. 
If $A(i,j,k)$ is $0$ or $1$, then necessarily $A(i,j,k)=B(i,j,k)$. 
So wherever $A(i,j,k)\neq B(i,j,k)$, there holds $A(i,j,k) = \frac{1}{2}$.

Consider the graph $G=G(A)$ whose vertices are the $\frac 12$ entries of $A$, where two vertices are adjacent iff 
they are on the same line. Since
$A$ is tristochastic, it follows that $B(i,j,k) + B(i',j',k') = 1$ for every two neighbors
$(i,j,k)$ and $(i',j',k')$ in $G$.
Specifically, if $B(i,j,k)=\frac{1}{2}+\epsilon$, then $B(i',j',k') =\frac{1}{2}-\epsilon$. 
Consequently, the connected component of $G$ which contains the vertices $(i,j,k)$ 
and $(i',j',k')$ is bipartite. 
The color of a vertex is determined according to whether the $B$ entry is $\frac 12 \pm \epsilon$. However, it is easy to verify that $G$ is connected and not bipartite, which proves our claim.

\subsection{A scheme for constructing vertices}

The above example suggests a construction for vertices of $\Omega^{(2)}_n$. Let
$A$ be an order-$n$ tristochastic array whose support consists of exactly two $\frac{1}{2}$ entries in each line. The graph $G = G(A)$ defined as above is 3-regular and has $2 n^2$ vertices. As we now show, $A$ is a vertex of $\Omega^{(2)}_n$ iff no connected component of $G$ is bipartite.

Indeed, suppose that $G$ has a bipartite connected component with parts $P$ and $Q$. Let $\Delta$ be the $[n]^3$ array with $\pm 1$ entries at the elements of $P, Q$ respectively and $0$ everywhere else. 
Note that every line of $\Delta$ sums to zero. To see that $A$ is not a vertex, note that $A = \frac{X+Y}{2}$, where $X, Y = A \pm \frac{1}{2}\Delta$ are clearly tristochastic.

Conversely, suppose that $A = \alpha B + (1-\alpha) C$ 
with $1 > \alpha > 0$ and $B \ne C$ in $\Omega^{(2)}_n$ is not a vertex.
The same consideration that worked for the above example shows that the relevant
component of $G$ is bipartite.

This discussion suggests that we construct $A$ so that no connected component 
of $G(A)$ is bipartite. This shouldn't be too hard, since $G$ is 3-regular. 
Indeed, we suspect (but we still cannot show) that a typical 
tristochastic array with two $\frac{1}{2}$'s in each line is a vertex. 
This idea still yields the following lower bound on the number of 
vertices of $\Omega^{(2)}_n$.

\begin{theorem}
The polytope $\Omega^{(2)}_n$ has at least $L_n^{\frac{3}{2}-o(1)}$ vertices.
\label{lower_bound}
\end{theorem}

\section{Proof of theorem \ref{lower_bound}}

\subsection{The construction}

Let $I=(i_1, i_2,\ldots,i_n)$ and $J=(j_1, j_2,\ldots,j_n)$ be two permutations of $[n]$. We let $H(I, J) := \{(i_1,j_1), (i_2,j_1), (i_2, j_2), \ldots, (i_{n}, j_{n}), (i_1, j_{n})\}$ and call such a collection of index pairs an $H$-{\em cycle}. Note that $H(I, J)=H(I', J')$ where $I'=(i_2, i_3,\ldots,i_n, i_1)$ and $J'=(j_2, j_3,\ldots,j_n, j_1)$. Likewise, $H(I, J)$ remains unchanged if we reverse the order of the $i_{\nu}$ and the $j_{\nu}$. Consequently, the number of $H$-cycles is $\frac 12 n! (n-1)!$.

\begin{definition}
Let $n$ be an even integer. A double Latin square is an 
$n \times n$ matrix with entries from $\{1, ... ,\frac{n}{2}\}$ 
where each symbol appears exactly \textbf{twice} in each row and column.
\end{definition}

We say that a double Latin square $X$ is \textit{Hamiltonian} if the indices of the $k$-entries of $X$ constitute an $H$-cycle for every $k \in \{1, ... ,\frac{n}{2}\}$. (This explains the choice of the term $H$-cycle).

Let $A$ be a $t \times t$ matrix and let $\sigma \in S_t$ 
be a permutation. We denote by $\sigma(A)$ the matrix obtained from $A$ by applying $\sigma$ to its rows. We need the following result from~\cite{DLS}:

\begin{proposition}
Let $A,B$ be two order $\frac{n}{2}$ Latin squares and let 
$\sigma \in \mathbb{S}_{\frac{n}{2}}$ be a cyclic permutation. Then the block matrix
\[  X = \left( \begin{array}{ccc} 
 A &\vline & B \\   
 \hline
 \sigma(A) & \vline &  B  \\   
 \end{array} \right).
\]
is a Hamiltonian double Latin square.
\end{proposition}

It follows that the number of Hamiltonian order-$n$ double Latin squares
is at least $(\frac{n}{2} -1)! \cdot L_{\frac{n}{2}}^2 = ((1+o(1))\frac{n}{2 e^2})^{\frac{n^2}{2}}$.

We want to construct a tristochastic array $A$ with exactly two 
$\frac{1}{2}$'s in each line, in such a way that 
$G(A)$ is non-bipartite and connected (and therefore $A$ is a vertex).

The idea is to use a Hamiltonian double Latin square $X$ to define the top $\frac{n}{2}$ layers of $A$. We use the fact that $X$ is Hamiltonian to complete $A$ in such a way that $G(A)$ is connected, and then ``plant" an odd cycle in $G(A)$ to ensure that $G(A)$ isn't bipartite.

Given a Hamiltonian double Latin square $X$, we use it as the ``topographical map" of the top $\frac{n}{2}$ layers of $A$. Namely, $A(i,j,k) = \frac{1}{2} \Leftrightarrow X(i,j)=k$. 
Let us observe the subgraph of $G(A)$ spanned by the entries of $A$ that reside in 
these top layers. Every positive entry $A(i,j,k) = \frac{1}{2}$ comes from $X(i,j)=k$, and 
$X$ has exactly two $k$ entries in each line.
Therefore this subgraph of $G(A)$ is $2$-regular. Moreover, since $X$ is Hamiltonian, for every 
$1 \leq k \leq \frac{n}{2}$ the vertices of $G(A)$ that correspond to $\mbox{supp}(A_k)$ constitute a cycle of length $2n$. In other words, the subgraph of $G(A)$ corresponding to the entries of the top half of $A$ is the disjoint union of $\frac{n}{2}$ cycles of length $2n$.

At this point, there are two $\frac{1}{2}$ entries in every line that resides in one of the 
top $\frac{n}{2}$ layers of $A$, and a single $\frac{1}{2}$ entry in every shaft.

We turn to define the next layer, $A_{\frac{n}{2}+1}$. Our purpose is to choose the $\frac 12$ entries in this layer so as to form a single cycle of length $2n$. The vertices of this subgraph should also be connected to each of the cycles in the top $\frac{n}{2}$ layers. Clearly, if we manage to accomplish this task, then the part of $G(A)$ that is already revealed is connected. Furthermore, note that every shaft contains a positive entry in the top half of $A$. Therefore, $G(A)$ will remain connected regardless of our choices in the lower layers of $A$.

In order to achieve our goals concerning $A_{\frac{n}{2}+1}$, we want to find $\frac{n}{2}$ index pairs 
$(i_1,j_1), ... ,(i_{\frac{n}{2}},j_{\frac{n}{2}})$ 
such that $X(i_l,j_l)=l$ for all $1\leq l\leq \frac{n}{2}$ and no two of them share a row or a column.  
We find such pairs successively as follows: Suppose that, for some $k< \frac{n}{2}$, we already have $k$ pairs
$(i_1,j_1),...,(i_k,j_k)$ with 
$X(i_1,j_1) = 1, ... ,X(i_k,j_k) = k$ and no two pairs share a row or column. 
We claim that there is an additional pair $(i_{k+1},j_{k+1})$ that does not
share a row or column with any of the above index pairs, and
$X(i_{k+1},j_{k+1})=k+1$. Since $X$ is a double Latin square, every row and column of $X$ has exactly two elements that equal $k+1$. Therefore at most $4k$ of these entries share a row or column with a previous pair. But $2n > 4k$, so that such an index pair $(i_{k+1},j_{k+1})$ must exist.

We choose $\frac{n}{2}$ more pairs of indices $(i_{\frac{n}{2}+1},j_{\frac{n}{2}+1}), ... ,(i_n,j_n)$ in such a way that no two pairs of $(i_1,j_1),...,(i_n,j_n)$ share a row or a column. 

It is possible to rename, if necessary, the set of chosen pairs $\{(i_{\alpha},j_{\alpha})|\alpha=1,\ldots,n\}$ as $\{(\nu,\tau_{\nu})|\nu = 1,\ldots,n\}$ for some permutation $\tau \in S_n$. Let $P$ be the permutation matrix of $\tau$. We next select a permutation $\sigma \in S_n$ whose permutation matrix $P'$ is such that $P + P'$ consists of a single cycle. (We note that given $\tau$, there are exactly $(n-1)!$ possible choices for $\sigma$).
We achieve our aim by setting $A_{\frac n2 +1} := \frac{1}{2}(P + P')$.

The purpose of our choices for $A_{\frac{n}{2}+2}$ is to introduce an odd 
cycle into $G(A)$. This odd cycle must use elements from the top half of $A$.
Additionally, the indices of the $\frac{1}{2}$ entries in $A_{\frac{n}{2}+2}$ must avoid all index pairs used in $A_{\frac{n}{2}+1}$, so as not to create a shaft with three $\frac{1}{2}$ entries.

To this end, we seek two vertices $x=(x_1,x_2,k)$ and $y=(y_1,y_2,k)$ with $x_1 \neq y_1$ and $x_2 \neq y_2$ that are connected by a path of odd length in the part of $G(A)$ constructed so far. The construction of $A_{\frac{n}{2}+2}$ will yield a length four path between $x$ and $y$, ensuring that $G(A)$ is not bipartite. This path will have the form $x,x',w,y',y$ where $x'=(x_1,x_2,\frac{n}{2}+2), y'=(y_1,y_2,\frac{n}{2}+2)$ and $w$ is either $(x_1,y_2,\frac{n}{2}+2)$ or $(y_1,x_2,\frac{n}{2}+2)$.

A simple counting argument shows the feasibility of this construction. Two vertices from the same layer can serve as $x$ and $y$ if their distance in that layer is odd and $\ge 3$. There are $\Omega(n^2)$ such pairs in every layer with a total of $\Omega(n^3)$ such candidate pairs. On the other hand, as we show below, only $O(n^2)$ such pairs are ruled out, so at least for large $n$ a good choice of such $x, y$ must exist. 

The reason that an entry cannot play the role of $x$ is that its shaft meets $\mbox{supp}(A_{\frac{n}{2}+1})$. There are $O(n)$ vertices in $x$'s layer which might serve as $y$, and $\mbox{supp}(A_{\frac{n}{2}+1})$ has cardinality $2n$, so only $O(n^2)$ pairs $x, y$ get ruled out for this reason. It remains to see how the pair $x=(x_1,x_2,k)$ and $y=(y_1,y_2,k)$ can be disqualified when both $x$'s and $y$'s shaft do not meet $\mbox{supp}(A_{\frac{n}{2}+1})$. This can happen only if both $(x_1,y_2,\frac{n}{2}+2)$ and $(y_1,x_2,\frac{n}{2}+2)$ are unavailable to us, namely $A(x_1,y_2,\frac{n}{2}+1) = A(y_1,x_2,\frac{n}{2}+1) = \frac 12$. There are only $O(n^2)$ such instances, one per each pair of vertices in the $2n$-cycle residing in $A_{\frac{n}{2}+1}$.

By doing these computations carefully, one shows that already for $n \geq 10$ there must exist a good pair for the above argument.  

Next we need to complete $\mbox{supp}(A_{\frac{n}{2}+2})$. We are currently committed to three elements and $2n-3$ more $\frac{1}{2}$ entries need to be chosen, so that altogether there are exactly two in each row and column. The locations that must not be chosen are those in the ``shadow" of $\mbox{supp}(A_{\frac{n}{2}+1})$. It is easily seen that we need the following simple graph-theoretic claim.

\begin{proposition}
Let $G=(L,R,E)$ be a $(n-2)$-regular bipartite graph with $|R|=|L|=n \ge 6$ and let $M$ be a path of length $3$ in $G$. Then there is a $2$-factor in $G$ which contains the three edges of $M$.
\end{proposition}

\begin{proof}
Let $M=x_1, x_2, x_3, x_4$. A bipartite graph with sides of size $k$ and degrees $\ge k/2$ has a perfect matching. Let $\Phi$ be a perfect matching in $G \setminus \{x_1, x_2, x_3, x_4\}$. Next let $\Psi$ be a perfect matching in $G \setminus \{x_2, x_3\} \setminus \Phi$. The desired $2$-factor is $\Phi \cup \Psi \cup \{(x_1, x_2),  (x_2, x_3), (x_3, x_4)\}$.
\end{proof}

To recap, the graph $G(A)$ is connected, it contains an odd cycle, and these properties are retained regardless of how the remaining $\frac{n}{2}-2$ layers are completed.

The remaining layers are constructed as follows. 
Let $K$ be an $n \times n$ matrix where $K(i,j) = 1$ or $0$ according to whether the shaft $A(i,j,\cdot)$ has one or two $\frac{1}{2}$ entries.
Each row and column of $K$ has $n-4$ one-entries. In other words, $K$ is the adjacency matrix of an $(n-4)$-regular bipartite graph which, therefore, has a $2$-factor. This process can be completed layer by layer. This is just an existential argument and we next turn to estimate the number of ways in which our construction can be realized.

To this end we will multiply the number of ways to construct the 
top half and the appropriate number for the bottom half. As stated above, there are $L_n^{\frac 12 +o(1)}$ ways to construct the top half. 
The estimate for the bottom $\frac{n}{2}-2$ layers is a slight variation on van Lint and Wilson's~\cite{VL+W} approximate enumeration of Latin squares. By the van der Waerden bound~\cite{Fal, Eg}, a $k$-regular $(n,n)$ bipartite graph $H$ has at least $\left((1+o(1)\frac ke\right)^n$ perfect matchings. By the same argument, there are at least $\left((1+o(1)\frac{k-1}e\right)^n$ ways to complete a perfect matching in $H$ to a $2$-factor. The product of these two numbers is an overcount, since every cycle in the $2$-factor can be split in two ways between the first and second $1$-factors. Consequently, $H$ has at least
\[
\left((1+o(1))\frac{k(k-1)}{e^2\sqrt 2}\right)^n
\]
$2$-factors. 

We think of $K$ as the adjacency matrix of such an $H$, and each layer is just a $2$-factor supported by $K$. With each choice, the edges of the chosen $2$-factor are removed from $H$, which goes from being $d$-regular to $(d-2)$-regular. This yileds the following lower bound on the number of choices:
$$ \prod_{2\leq k \leq n-4, ~k \small{\mbox{~is even}}}{\left((1+o(1))\frac{k(k-1)}{e^2\sqrt 2}\right)^n} = $$
$$(n-4)!^{n} \cdot \left(\frac{1+o(1)}{e^2\sqrt 2}\right)^{n(n-4)/2} = \left( (1+o(1)) \frac{n}{2^{\frac{1}{4}}e^2}  \right)^{n^2} = L_n^{1-o(1)}.$$

The product of the bound for the top half and the bound for the bottom half yields a total of 
$L_n^{\frac{3}{2}-o(1)}.$

\section{A variation on the theme}\label{variation}

Here is another natural extension of the notion of doubly stochastic matrices. Namely, let $\Sigma_n^{(d)}$ to be the set of all $[n]^{d+1}$ arrays of nonnegative reals such that the entries in each coordinate hyperplane sum to one. The collection of such arrays clearly constitutes a convex polytope. Our goal is to investigate the vertex set of this polytope. Let us define $T_n^{(d)}$ as the collection of all $[n]^{d+1}$ arrays of zeros and ones with a single one in each coordinate hyperplane. It is clear that $T_n^{(d)}$ is included in the vertex set of $\Sigma_n^{(d)}$. 
In view of the Birkhoff-von Neumann theorem it is natural to see how many of these vertices belong to $T_n^{(d)}$.

There is a natural bijection between tuples $(\sigma_1, ... ,\sigma_d) \in \mathbb{S}_n^d$ and members $A \in T_n^{(d)}$ which is given by 
$ A(i,\sigma_1(i), ... ,\sigma_d(i)) = 1 $ for all $1 \leq i \leq n$.
In particular $|T_n^{(d)}| = (n!)^d$. 

As mentioned, every member of $T_n^{(d)}$ is a vertex of $\Sigma_n^{(d)}$, and we ask whether this polytope has any additional vertices. As it happens, such vertices are easy to construct. Here is the smallest example:

\[ 
A_1 = \left[\begin{array}{cc}
\frac{1}{2} & 0  \\
0 & \frac{1}{2}  \\
\end{array}\right],
A_2 = \left[\begin{array}{cc}
0 & \frac{1}{2}  \\
\frac{1}{2} & 0  \\
\end{array}\right]
\]

Clearly $A \in \Sigma_2^{(2)}$. We now consider the graph $\bar{G}(A)$ 
with vertex set $\mbox{supp}(A)$ with an edge between every two vertices that lie in the same coordinate hyperplane. As in Section~\ref{basic}, we show that $A$ is a vertex by observing that $\bar{G}(A)$ has no bipartite connected component. In the present case, $\bar{G}=K_4$.

Our general construction is similar in nature to this example. We first construct an $n \times n$ matrix $M$ with entries from $[n]$ in which every row and column contains exactly two nonzero entries and where each integer in $[n]$ appears exactly twice in $M$. We view $M$ as a way to encode $A$ as follows: $M(i,j)=k$ for some $k \neq 0$ says that $A(i, j, k)=1/2$ and $A(i, j, k')=0$ for all $k' \neq k$. Also $M(i,j)=0$ means that $A(i, j, l)=0$ for all $l$. It is not hard to verify that if the graph $\bar G$ corresponding to $M$ is connected and non-bipartite, then $A$ is a vertex of $\Sigma_n^{(2)}$.

We now turn to construct many such matrices $M$ and thus generate many vertices for $\Sigma_n^{(2)}$ which are not in $T_n^{(2)}$. Let $$H=\{(i_1,j_1), (i_2,j_1), (i_2, j_2), \ldots, (i_{n}, j_{n}), (i_1, j_{n})\}$$ be an $H$-cycle and let
$$ M(i_1,j_1) =  M(i_2,j_2) = 1\mbox{~ and~} M(i_2,j_1) = 2 .$$
The remaining entries of the $H$-cycle $M(i_{\alpha}, j_{\alpha})$ and $M(i_{\alpha+1}, j_{\alpha})$ are filled arbitrarily with the elements of the multiset $\{2,3,3,4,4,\ldots,n,n\}$. Note that the resulting graph $\bar{G}$ is connected due to the $H$-cycle. It also contains the triangle $\{(i_1,j_1, 1), (i_2,j_2, 1), (i_2,j_1, 2)\}$. There are $\frac 12 n!(n-1)!$ choices for $H$ and
and $\frac{(2n-3)!}{2^{n-2}}$ ways to map the multiset to the nonzero entries of $M$. Altogether, this construction yields more than $n!^4 > \left(T_n^{(2)}\right)^2$ vertices of $\Sigma_n^{(2)}$. It follows that $T_n^{(2)}$ constitutes a vanishingly small subset of this vertex set.

\section{Conjectures and some experimental results}

This paper raises many open questions. Here are several of them:
\begin{itemize}
\item
Get a better estimate for the number of vertices of $\Omega^{(2)}_n$.
\item
The analogous question for $\Omega^{(d)}_n$ with $d > 2$ seems completely open at this writing.
\item
The polytope $\Omega^{(2)}_n$ is defined by requiring that one-dimensional subsets of the array sum to one. In the definition of $\Sigma_n^{(2)}$ this is required of two-dimensional subsets. For larger $d$ this suggests a whole range of possible polytopes to consider, depending on which sets of entries sum to $1$. 
\end{itemize} 

If we knew the support size of vertices in $\Omega^{(d)}_n$, we could make progress on these questions. By standard linear programming arguments, every vertex of $\Omega^{(d)}_n$ has at least $\mbox{aff-dim}(\Omega^{(d)}_n)$ zero coordinates. Since $\mbox{aff-dim}(\Omega^{(d)}_n) = (n-1)^{d+1}$, every vertex has support size at most $n^{d+1} - (n-1)^{d+1}\le (d+1) \cdot n^{d}$.  

It follows that $\Omega^{(d)}_n$ has at most
$\binom{n^{d+1}}{(d+1) n^{d}}\le \left( \frac{ne}{d+1} \right)^{(d+1) n^{d}}$ 
vertices. In particular, $\Omega^{(2)}_n$ has fewer than 
$ n^{3 n^2}$ vertices. If we knew, say, that a typical vertex
of $\Omega^{(2)}_n$ has support size $\le \alpha n^2$ vertices, we could conclude
that it has at most  $n^{(1+o(1))\alpha n^2}$ vertices.

We have conducted some numerical experiments to get a sense of the numbers. Using
linear programming tools, it is possible to find the vertex that maximizes a randomly chosen linear objective function. Needless to say, this distribution on the vertices is by no means uniform. We nevertheless hope that our experiments do tell us something meaningful about the properties of typical vertices. We selected the coordinates in the objective function independently from normal distribution. The average value of $\alpha$ in these experiments seems to increase slowly with $n$. We don't know whether the typical support size of a vertex converges to $3n^2$ or to $\alpha n^2$ for some $\alpha < 3$.

\end{document}